\documentclass[12pt,oneside,reqno]{amsart}
\usepackage{amssymb}
\usepackage[active]{srcltx}
\usepackage{a4wide}
\usepackage[section]{placeins}
\usepackage{amsmath,amsthm}
\usepackage{amssymb}
\usepackage{amstext}
\everymath{\displaystyle}
\usepackage{cite}
\usepackage{epsfig}
\usepackage{float}
\usepackage{enumerate}
\usepackage{color}
\usepackage{amssymb}
\usepackage[pagewise]{lineno}
\usepackage[dvipsnames]{xcolor}
\usepackage{xcolor}
\usepackage{soul}
\usepackage{ragged2e}
\newtheorem{theorem}{Theorem}[section]
\newtheorem{lemma}[theorem]{Lemma}
\newtheorem{proposition}[theorem]{Proposition}
\theoremstyle{definition}

\newtheorem{corollary}[theorem]{Corollary}
\theoremstyle{remark}
\newtheorem{remark}[theorem]{Remark}
\newtheorem{note}[theorem]{Note}
\usepackage[linkcolor=blue, urlcolor=blue, citecolor=blue,
colorlinks, bookmarks]{hyperref}

\newcommand{\be}{\begin{equation}}
\newcommand{\ee}{\end{equation}}

\title[]{Dimension of Inhomogeneous Sub-Self-Similar Sets}

\title{Dimension of Inhomogeneous Sub-Self-Similar Sets}

\author{Shivam Dubey}
\address{Department of Applied Sciences, IIIT Allahabad, Prayagraj, India 211015}
\email{rss2022509@iiita.ac.in}
\author{Saurabh Verma}

\address{Department of Applied Sciences, IIIT Allahabad, Prayagraj, India 211015}

\email{saurabhverma@iiita.ac.in}



\begin{document}

\maketitle

\vspace{-1em}

\begin{abstract}
In this paper, we introduce the concept of inhomogeneous sub-self-similar (ISSS) sets, building upon the foundations laid by Falconer (Trans. Amer. Math. Soc. 347 (1995) 3121–3129) in the study of sub-self-similar sets and drawing inspiration from Barnsley's work on inhomogeneous self-similar sets (Proc. Roy. Soc. London Ser. A 399 (1985), no. 1817, 24). We explore a range of ISSS sets and present a method for constructing them. We also investigate the upper and lower box dimensions of ISSS sets and discuss the continuity of the Hausdorff dimension.
\end{abstract}


\section{Introduction}\label{section 1} 
A wide variety of fractals exhibit self-similarity, that is, they consist of infinitely many smaller replicas that resemble the whole. Hutchinson \cite{Hut} provided a comprehensive framework for self-similar sets and constructed them mathematically using \textit{iterated function systems (IFS)}. The construction of self-similar sets via IFS has been extensively studied and popularized by Barnsley and his collaborators \cite{MF1, MF2, MF3}. Barnsley \cite{MF1} introduced the concept of fractal interpolation functions (FIFs), and since then, many researchers have worked on FIFs and their applications. For further details, interested readers may see \cite{SS5, Celik, KVB1, LCP1, M2, Ri2, Ruan3, Ruan4, Verma21}. Barnsley and Demko \cite{MF3} extended the IFS framework by introducing \textit{inhomogeneous IFSs}, which include the classical IFS together with a compact set called a \textit{condensation set}. Inhomogeneous IFSs generalize the classical notion of IFS by reducing to the standard self-similar case when the condensation set is empty. Fraser \cite{Fraser1} made significant and notable contributions in this area, and Snigireva \cite{Sn}. For some related results on inhomogeneous IFSs, we refer the reader to \cite{B, K, Bu, Frase}. Furthermore, Mauldin and Williams \cite{Rd} introduced graph-directed IFS (GD-IFS), which generalizes the classical IFS framework. In this regard, Boore and Falconer \cite{bFalc} provided an example of a graph-directed fractal that a classical IFS cannot generate. More recently, Dubey and Verma further generalized GD-IFS by introducing the concept of \textit{inhomogeneous GD-IFS}; see \cite{DV1, DV2, Nussbaum1} for details. In \cite{Fal1}, Falconer introduced the notion of \textit{sub-self-similar (SSS) sets} by replacing equality in the self-similar equation with set inclusion. Computing the exact dimension of a fractal set remains a challenging problem in fractal geometry, and various methods have been developed to estimate the fractal dimensions. Chandra and Abbas \cite{Chana} computed the fractal dimension of mixed Riemann--Liouville integrals using the $\delta$-covering method. Falconer \cite{Fal} studied the fractal dimensions using the Fourier transform method and also explored potential-theoretic techniques and the mass distribution principle. Priyadarshi \cite{Pr1} estimated better lower bounds for the Hausdorff dimension of sets of complex continued fractions using operator-theoretic techniques. Chandra and Abbas \cite{SS5} investigated the fractal dimension of fractal functions using function space techniques. Hochman \cite{MHOCHMAN} studied the Hausdorff dimension of self-similar sets and measures on the real line using entropy methods and established an inverse theorem for the entropy of convolution. Shmerkin \cite{Shmerkin} extended this work by studying the $L^q$-dimension of dynamically driven self-similar measures. Achour and Selmi \cite{Ach} studied the general packing dimension of typical attractors and measures. Li \textit{et al.} \cite{Acho} investigated the generalized fractal dimensions of the graphs of functions obtained via the sum and product of continuous functions. Cui \textit{et al.} \cite{Cui} discussed the variational principle relating packing topological pressure to the measure-theoretic pressure of Borel probability measures associated with non-autonomous IFSs. For a given sequence of real numbers, Yu \textit{et al.} \cite{Byu} examined the relationship between convergence and divergence of series and fractal dimensions. Following Falconer's work \cite{Fal1} and using the concept of inhomogeneous self-similar sets, we introduce and construct \textit{inhomogeneous sub-self-similar sets (ISSS)}. Motivated by Fraser's work \cite{Fraser1}, we estimate the fractal dimension of these sets using the $\delta$-covering method.

\par The outline of the paper is as follows: Section $2$ provides the preliminaries. In Section $3$, we present examples of ISSS sets and discuss a method to construct them, while Sections $4$ and $5$ are devoted to estimating the fractal dimension of these sets and to examining the continuity of the Hausdorff dimension for ISSS sets, respectively. In Section $6$, we see the self-similar structure on the product of inhomogeneous IFSs. Finally, Section $7$ concludes the paper with remarks and outlines for potential future research directions.
\par

\section{Preliminaries}\label{Section 2}
In this section, we recall the notion of sub-self-similar sets and introduce the concept of inhomogeneous sub-self-similar sets.

Let $(X, d)$ be a complete metric space, and let $f_1, f_2, \dots, f_N$ be similarity mappings on $X$. A unique non-empty compact set $A$ is called the \emph{self-similar set} associated with the iterated function system (IFS) $\mathcal{I} := \{X; f_1, f_2, \dots, f_N\}$ if $A$ satisfies the following equation, known as the \emph{self-similar equation}:
\begin{equation}\label{eq1}
A = \bigcup_{i=1}^N f_i(A).
\end{equation}

Falconer \cite{Fal1} generalized the concept of self-similar sets by replacing equality in the self-similar equation with set inclusion. A non-empty compact set $F \subseteq X$ is called a \emph{sub-self-similar set (SSS set)} if it satisfies:
\begin{equation}\label{eq2}
F \subseteq \bigcup_{i=1}^N f_i(F).
\end{equation}

Given a compact set $C \subseteq X$, referred to as the \emph{condensation set}, a unique non-empty compact set $A_C$ is called the \emph{inhomogeneous self-similar set} associated with the IFS $\mathcal{I} := \{X; f_1, f_2, \dots, f_N, C\}$ if it satisfies the following equation, known as the \emph{inhomogeneous self-similar equation}:
\begin{equation}\label{eq3}
A_C = \bigcup_{i=1}^N f_i(A_C) \cup C.
\end{equation}

Analogously, we define \emph{inhomogeneous sub-self-similar sets (ISSS sets)}, which extend the concept of inhomogeneous self-similar sets by replacing equality in \eqref{eq3} with set inclusion. A non-empty compact set $F_C \subseteq X$ is called an ISSS set if it satisfies:
\begin{equation}\label{eq4}
F_C \subseteq \bigcup_{i=1}^N f_i(F_C) \cup C.
\end{equation}

We briefly recall the separation conditions needed in the upcoming sections. The IFS $\mathcal{I}$ satisfies the \emph{strong separation condition (SSC)}, if 
$f_{i}(A) \cap f_j(A) = \emptyset ~~ \text{for all} ~ i \neq j.$
If there exists a non-empty bounded open set $U$ such that 
$\bigcup_{i=1}^N f_i(U) \subseteq U ~~ \text{and}~~ f_i(U) \cap f_j(U)= \emptyset ~~\text{for each}~~ i \neq j.$
Then, we say that IFS $\mathcal{I}$ satisfies the \emph{open set condition (OSC)}. Further, we say it satisfies the \emph{strong open set condition (SOSC)}, if the open set $U$ has a non-empty intersection with the attractor $A$. 
For the inhomogeneous IFS $\mathcal{I}_C$, if the homogeneous IFS satisfies the OSC and $C \subset U$, then we say that the inhomogeneous IFS $\mathcal{I}_C$ satisfies the \emph{inhomogeneous open set condition (IOSC)}. Further, if $U$ has a non-empty intersection with the corresponding inhomogeneous attractor $A_C$, then IFS $\mathcal{I}_C$ satisfies the \emph{inhomogeneous strong open set condition (ISOSC)}.\\

Recall that the upper and lower box-dimensions of a non-empty compact set $F \subset \mathbb{R}^n$ are given by
 $$\overline{\dim}_{B}(F) = \limsup_{\delta \to 0^+ } \frac{\log N_{\delta}(F)}{-\log \delta}, ~\text{and} ~\underline{\dim}_{B}(F) = \liminf_{\delta \to 0^+ } \frac{\log N_{\delta}(F)}{-\log \delta}, ~\text{respectively},$$
 where $N_{\delta}(F)$ represents the minimum number of sets of diameter $\delta$ covering $F$,
Furthermore, if the above two limits are equal, then the box dimension of $F$ is given by 
$${\dim}_{B}(F) = \lim_{\delta \to 0^+ } \frac{\log N_{\delta}(F)}{-\log \delta}.$$

For $0\le s \le n$ the $s$-dimensional Hausdorff measure of a set $F \subseteq \mathbb{R}^n$ is defined by 
$$H^s(F)= \lim_{\delta \to \infty} H^s_\delta(F),$$
where $H^s_\delta(F)= \inf\big\{\sum_{i=1}^{\infty} |U_i|^s : F \subseteq \bigcup_{i=1}^{\infty} U_i ~\text{and}~ |U_i| \le \delta\big\}$ and $|U_i|$ denote the diameter of $U_i \subseteq \mathbb{R}^n.$
The Hausdorff dimension of $F \subseteq \mathbb{R}^n$ is the threshold value at which the Hausdorff measure jumps from infinity to zero, i.e., the Hausdorff dimension of $F$ is the number $0 \le d \le n$ such that $H^s(F)= \infty,$ if $0 \le s < d$ and $H^s(F)=0,$ if $d <s \le n.$
\subsection{Code Space}
In this subsection, we see the \emph{code space (symbolic space)} associated with finite symbols $N$ of the set $I = \{1,2,\ldots, N\}$.
 Let $I^\infty$ denote the set of all infinite sequences made up of $\{1, \ldots, N\}$, that is,
\[
I^\infty := \left\{ \{\omega_k\}_{k=1}^{\infty} :\, \omega_k \in \{1, \ldots, N\} \right\}.
\]
Any element $\omega \in I^\infty,$ we say $\omega$ is a string of length infinite and write $\omega := \omega_1 \omega_2 \dots \omega_k \omega_{k+1} \dots$, where $\omega_k$ represents the $k$-th term of the sequence.\\
Let $\rho_1, \rho_2,\ldots,\rho_n$ be real numbers such that $0< \rho_i <1$ for all $1 \le i \le N$, then define 
$$d_\infty(\omega, \omega')= \rho_1 \cdots \rho_k,$$
if the infinite strings $\omega = \omega_1 \omega_2 \dots \omega_k \omega_{k+1} \dots$ and $\omega' = \omega'_1 \omega'_2 \dots \omega'_k \omega'_{k+1} \dots$ agree up to the $k$-th term and differ at the $(k+1)$-th term. If the sequences are identical, then $d_\infty(\omega,\omega') = 0$; and if they differ at the first term, then $d_\infty(\omega, \omega') = 1$. The map $d_\infty$ is a metric on $I^\infty$, and $(I^\infty, d_\infty)$ is a compact, complete, and totally disconnected space. \\

 Let $\mathcal{I}$ be an IFS consisting of similarity mappings $\{f_1, \ldots, f_N\}$, and let $A$ be the unique self-similar set associated with $\mathcal{I}$. There is a well-known map called the \emph{coding map} $\Theta: I^\infty \to A$ defined by
\[
\Theta(\omega) := \lim_{n \to \infty} f_{\omega_1} \circ \cdots \circ f_{\omega_n} (x) = \bigcap_{n=1}^{\infty} f_{\omega_1} \circ \cdots \circ f_{\omega_n}(A), \quad \text{for all}~ \omega = \omega_1 \omega_2 \dots \in I^{\infty},
\]
which associates each infinite string $\omega$ with a unique point in the attractor $A$.\\
\begin{note}
The coding map $\Theta$ is onto but not necessarily one-to-one. If it is one-to-one, then it becomes a homeomorphism onto the attractor. It is easy to observe that when the IFS satisfies the SSC, the coding map is a homeomorphism onto the attractor. 
\end{note}
Let $\omega_{|_n}:= \omega_1 \omega_2 \ldots \omega_n$ represents the restriction of $\omega \in I^{\infty}$ up-to the $n$th term, then define $I^n:=\{\omega_{|_n}: \omega \in I^{\infty}\}$ and $I^*:= \bigcup_n I^n.$ We say $\omega= \omega_1 \omega_2 \ldots \omega_n \in I^*$ is a string of length $n$ and the length of the string $\omega$ is denoted by $|\omega|=n$. For two strings $\omega:= \omega_1\omega_2\ldots \omega_n \in I^n$ and $\omega':= \omega'_1 \omega'_2 \ldots \omega'_m \in I^m$, we define concatenation of $\omega$ and $\omega'$ by $\omega'':= \omega_1\ldots \omega_n \omega'_1 \ldots \omega'_m \in I^{n+m}$ is a string of length $n+m.$\\

For any subset $S \subseteq I^\infty$ and a positive integer $n$, define
\[
S^n := \left\{ \omega_{|_n} = \omega_1 \dots \omega_n :\, \exists ~ \omega = \omega_1 \omega_2 \dots \omega_k \dots \in S \right\}, \quad \text{and} \quad S^* := \bigcup_{n} S^n.
\]

A compact subset $S \subseteq I^\infty$ is said to be \emph{closed under left shift} if for every $\omega := \omega_1 \omega_2 \dots \in S$, the sequence $\omega':= \omega_2 \omega_3 \dots$ also belongs to $S$. 

\begin{remark}
The convergence \(\Theta(\omega) = \lim_{n \to \infty} f_{\omega_1 \ldots \omega_n}(x)\) is independent of the choice of the point \(x \in X\).
\end{remark}

For any \(\omega := \omega_1 \dots \omega_n \in I^n\), we write
\[
f_{\omega} := f_{\omega_1} \circ f_{\omega_2} \dots \circ f_{\omega_n}, \qquad 
\rho_{\omega} := \prod_{i=1}^n \rho_{\omega_i}.
\]

Given $s\ge0$, consider
$$\sum_{\omega \in S^{m+n}} \rho_\omega^s \le \sum_{\omega \in S^m, \omega' \in S^n} \rho_{\omega \omega'}^s = \bigg(\sum_{\omega \in S^m} \rho_{\omega}^s \bigg) \bigg(\sum_{\omega' \in S^n} \rho_{\omega'}^s\bigg).$$
This shows that the sequence $\sum_{\omega \in S^m} \rho_{\omega}^s$ is a sub-multiplicative sequence. Hence, by the standard properties of the sub-multiplicative sequences, the limit $\tau(s) \equiv \lim_{k \to \infty} \bigg( \sum_{\omega \in S^k} \rho_{\omega}^s \bigg)^{1/k}$
exists with $0 \le \tau(s) < \infty,$ see \cite{Fal1}.

 Now we state important results of Falconer \cite{Fal1} that characterize SSS sets and determine their dimensions.

\begin{proposition}[see~\cite{Fal1}, Proposition~2.1]
\label{prop2.1}
Let $\mathcal{I} = \{(X, d), f_1, \ldots, f_N\}$ be an IFS. Then a non-empty compact set $F \subseteq X$ is an SSS set for $\mathcal{I}$ if and only if $F = \Theta(S)$, for some compact set $S \subseteq I^\infty$ that is closed under left shift. Here, $\Theta(\omega) = \lim_{n \to \infty} f_{\omega_1}\circ \cdots \circ f_{\omega_n}(x)$ for any $\omega = \omega_1 \omega_2 \dots \in S$ and any $x \in X$.
\end{proposition}

\begin{theorem}\label{thm51}(see, ~\cite[Theorem~ 3.5]{Fal1})
    Let $E$ be an SSS set for a family of contracting similarities $f_1,\ldots,f_N$. Let $s$ be the unique non-negative number such that 
    $\tau(s)\equiv \lim_{k \to \infty} \bigg( \sum_{i \in S^k} \rho_{i}^s \bigg)^{1/k} =1,$ then
    $$\dim_{H} E \le \underline{\dim}_B E \le \overline{\dim}_B E \le s.$$
    Furthermore, if it satisfies the OSC, then
    $$\dim_{H} E = \underline{\dim}_B E = \overline{\dim}_B E = s.$$
\end{theorem}

\section{Examples and Construction of Inhomogeneous Sub-Self-Similar Sets}

In this section, we present various examples of ISSS sets and provide one method for their construction.

\subsection*{Example 3.1 (Trivial Cases)}
Inhomogeneous self-similar sets and SSS sets are trivial examples of ISSS sets that satisfy Equation~\eqref{eq4}.

\subsection*{Example 3.2 (Union of ISSS Sets)}
Let $F_C$ be an ISSS set associated with the IFS $\{f_1, \ldots, f_N, C\}$, and let $F_{C'}$ be another ISSS set associated with the IFS $\{f'_1, \ldots, f'_{N'}, C'\}$. Then the union $F_C \cup F_{C'}$ is also an ISSS set associated with the IFS $\{f_1, \ldots, f_N, f'_1, \ldots, f'_{N'}, \\C \cup C'\}$.
Since
\begin{align*}
F_C &\subseteq \bigcup_{i=1}^N f_i(F_C) \cup C, \\
F_{C'} &\subseteq \bigcup_{i'=1}^{N'} f'_{i'}(F_{C'}) \cup C'.
\end{align*}
Thus,
\[
F_C \cup F_{C'} \subseteq \bigcup_{i=1}^N f_i(F_C \cup F_{C'}) \cup \bigcup_{i'=1}^{N'} f'_{i'}(F_C \cup F_{C'}) \cup (C \cup C').
\]

\subsection*{Example 3.3 (Augmenting ISSS Sets)}
Let $F_C$ be an ISSS set associated with the IFS $\{f_1, \ldots, f_N, C\}$. If $K \subseteq F_C$ is a compact subset and $f$ is a contracting similarity, then the set $F_C \cup f(K)$ is an ISSS set for the IFS $\{f_1, \ldots, f_N, f, C\}$. Since $f(K) \subseteq f(F_C)$, we have
\[
F_C \cup f(K) \subseteq \bigcup_{i=1}^N f_i(F_C \cup f(K)) \cup f(F_C \cup f(K)) \cup C.
\]

\subsection*{Example 3.4 (Boundary of an ISSS Set)}
Let $F_C$ be an ISSS set satisfying Equation~\eqref{eq4}. Then the topological boundary $\partial F_C$ of $F_C$ is also an ISSS set. 
 Any neighbourhood of $x \in \partial F_C$ contains points outside $F_C$, implying that $x \in \partial f_i(F_C).$ Also, there exists some $y \in F_C$ such that $f_i(y)=x$ for some $i.$ We claim that $y \in \partial F_C.$ Since $f_i$ is a contraction map, it is continuous and injective. Consequently, $y$ cannot be an interior point of $F_C$. Thus,  $y \in \partial F_C$ and $x \in f_i(\partial F_C),$ or $x \in \partial C \subseteq C$ completing the assertion.


\begin{remark}
Unlike inhomogeneous self-similar sets, ISSS sets are not unique. For example, consider the IFS $\mathcal{I} = \{f_1(x) = \frac{x}{2}, f_2(x) = \frac{x}{2} + \frac{1}{2}\}$ with a condensation set $C$. Both the set $\{0, 1\}$ and the set of all dyadic numbers in $[0, 1]$ are ISSS sets corresponding to $\mathcal{I}$ and $C$. In particular, if we consider $C=\{\frac{1}{2}\},$ then set 
$\{\frac{1}{2^n}: n \ge 1\} \cup \{1+\frac{1}{2^n}: n \ge 1\}$ is an ISSS with condensation set $C$ set but not an SSS set.
\end{remark}

\begin{remark}
Let $E$ be a bounded set that is not an SSS with respect to an IFS $\{f_i\}_{i=1}^N$. Then the set $C = \overline{E \setminus \bigcup_{i=1}^N f_i(E)}$ ensures that $E = \bigcup_{i=1}^N f_i(E) \cup C$, making $E$ an ISSS set. Moreover, for any $C' \supseteq C$, $E$ remains an ISSS set.
\end{remark}

\begin{remark}
In \cite{Mac}, McClure and Vallin provide examples of non-sub-self-similar sets, such as $E = \{0, 1, \frac{1}{2}, \frac{1}{3}, \ldots\}$, which is not an SSS set. As noted previously, there exists a condensation set $C$ such that $E$ becomes an ISSS set. Furthermore, for any compact set $E$, the minimal condensation set $C = \overline{E \setminus \bigcup_{i=1}^N f_i(E)}$ ensures that $E$ is an ISSS set.
\end{remark}
\begin{remark}
    Inhomogeneous self-similar sets are ISSS sets. Also, if we set the condensation sets $C$ to the empty set, then ISSS sets reduce to SSS sets, which are a generalization of self-similar sets. In this way, ISSS sets are a generalization of self-similar sets.
\end{remark}
\subsection{Construction of Inhomogeneous Sub-Self-Similar Sets}

We now present a constructive method for generating ISSS sets from a given SSS set and a condensation set.

\begin{theorem}
Let $E$ be an SSS set associated with the IFS $\mathcal{I}$ such that $\Theta(S) = E$, and let $C \subset X$ be a non-empty compact set. Define
\[
O_S := \bigcup_{{\omega} \in S^*} f_{{\omega}}(C).
\]
Then $ E\cup O_S$ is an ISSS set corresponding to the IFS $\mathcal{I}$ with condensation set $C$.
\end{theorem}
\begin{proof}
To prove that $E\cup O_S$ is an ISSS set, we first show that it is compact. Since $O_S$ is bounded, it suffices to prove that $E \cup O_S$ is closed. Let $\{x_n\}$ be a sequence in $E \cup O_S$ converging to some $x \in X$. If $x_n \in E$ for infinitely many $n$, then there exists a subsequence $\{x_{n_r}\}$ with $x_{n_r} \in E$. Since $E$ is closed, $\lim x_{n_r} = x \in E \subseteq E \cup O_S$. If $x_n \in E$ for only finitely many $n$, i.e., $\{x_n\}$ is eventually in $O_S$, then there exists $n_0$ such that $x_n \in O_S$ for all $n \geq n_0$. Hence, for each $n \geq n_0$, there exists a string $\omega_{|_n} \in S^*$ such that $x_n \in f_{\omega_{|_n}}(C)$.
Let us denote the length of any string $\omega := \omega_1 \omega_2 \dots \omega_k \in S^*$ by $|\omega|=k.$

\textbf{Case 1:} $\liminf |\omega_{|_n}| = k < \infty$. Then there exists a subsequence $\{\omega_{{|_n}_j}\}$ of $\{\omega_{|_n}\}$ such that $\omega_{{|_n}_j} \in S^k$ for all $j$, where $S^k$ is a finite set. Thus, there exists a string $t \in \{\omega_{{|_n}_j} \mid j \in \mathbb{N}\}$ and a further subsequence $\{\omega_{{|_n}_{j_p}}\}$ such that $\omega_{{|_n}_{j_p}} = t$ for all $p$. Then $x_{n_{j_p}} \in f_t(C)$, and since $f_t(C)$ is closed, $\lim x_{n_{j_p}} = x \in f_t(C) \subseteq O_S$.

\textbf{Case 2:} $\liminf |\omega_{|_n}| = \infty$. Then
\[
\begin{aligned}
\operatorname{diam}(f_{\omega_{|_n}}(C) \cup f_{\omega_{|_n}}(E)) &= \rho_{\omega_{|_n}} \operatorname{diam}(C \cup E) \\
&\leq \rho_{\max}^{|\omega_{|_n}|} \operatorname{diam}(C \cup E) \to 0 \quad \text{as } n \to \infty,
\end{aligned}
\]
since $\rho_{\max} = \max\{\rho_i\} < 1$.

Now, for each $n$, choose $y_n \in f_{\omega_{|_n}}(E)$. For positive integers $m$ and $n$,
\[
\begin{aligned}
|y_n - y_m| &\leq |y_n - x_n| + |x_n - x_m| + |x_m - y_m| \\
&\leq \operatorname{diam}(f_{\omega_{|_n}}(C) \cup f_{\omega_{|_n}}(E)) + |x_n - x_m| + \operatorname{diam}(f_{\omega_{|_m}}(C) \cup f_{\omega_{|_m}}(E)).
\end{aligned}
\]
Hence, $\{y_n\}$ is a Cauchy sequence and converges to some $y \in X$. Then,
\[
\begin{aligned}
|x - y| &\leq |x - x_n| + |x_n - y_n| + |y_n - y| \\
&\leq |x - x_n| + \operatorname{diam}(f_{\omega_{|_n}}(C) \cup f_{\omega_{|_n}}(E)) + |y_n - y| \to 0.
\end{aligned}
\]
Thus, $x = y$.

Since $y_n \in f_{\omega_{|_n}}(E)$ for each $n$, there exists $y' \in E$ such that $y_n = f_{\omega_1 \ldots \omega_n}(y')$. Let $\epsilon > 0$. Choose $n_1 \in \mathbb{N}$ such that for all $n \geq n_1$,
\[
|y_n - y| = |f_{\omega_1 \ldots \omega_n}(y') - y| < \epsilon / 2.
\]
Since $\Theta(S) = E$, we have $\lim_{n \to \infty} f_{\omega_1 \ldots \omega_n}(y') = z_{n_\epsilon} \in E$. Choose $n_2 \in \mathbb{N}$ such that for all $n \geq n_2$,
\[
|f_{\omega_1 \ldots \omega_n}(y') - z_{n_\epsilon}| < \epsilon / 2.
\]
Let $n' = \max\{n_1, n_2\}$. Then,
\[
\begin{aligned}
|y - z_{n_\epsilon}| &\leq |f_{\omega_1 \ldots \omega_n}(y') - y| + |f_{\omega_1 \ldots \omega_n}(y') - z_{n_\epsilon}| \\
&< \epsilon / 2 + \epsilon / 2 = \epsilon.
\end{aligned}
\]
Since $E$ is closed, $x = y = \lim z_{n_\epsilon} \in E \subseteq E \cup O_S$. Thus, $E \cup O_S$ is closed, and hence compact.

It remains to show that $E \cup O_S$ satisfies Equation~\eqref{eq4}.

Clearly,
\[
\begin{aligned}
E \cup O_S &\subseteq \bigcup_{i=1}^N f_i(E) \cup O_S \\
&\subseteq \bigcup_{i=1}^N f_i(E \cup O_S) \cup O_S \cup C.
\end{aligned}
\]

Now we show that $O_S \subseteq \bigcup_{i=1}^N f_i(O_S)$. Let $x \in O_S$. Then there exists $\omega_{|_n} \in S^*$ and $c \in C$ such that
\[
x = f_{\omega_{|_n}}(c) = f_{\omega_1 \ldots \omega_n}(c).
\]
Let $x' = f_{\omega_2 \ldots \omega_n}(c) \in O_S$. Then,
\[
x = f_{\omega_1}(x') \in \bigcup_{i=1}^N f_i(O_S).
\]

Thus,
\[
E \cup O_S \subseteq \bigcup_{i=1}^N f_i(E \cup O_S) \cup C.
\]
This completes the proof.
\end{proof}

\begin{theorem}
Let $E$ be an SSS set with $\Theta(S) = E$, and
\[
O_S = \bigcup_{\omega \in S^*} f_{\omega}(C).
\]
Then,
\[
\overline{O}_S = E \cup O_S.
\]
\end{theorem}

\begin{proof}
By the construction, $\overline{O}_S$ is the smallest closed set containing $O_S$. Since $E \cup O_S$ is closed and $O_S \subseteq E \cup O_S$, it follows that $\overline{O}_S \subseteq E \cup O_S$.

To prove the reverse inclusion, let $x \in E\cup O_S$ be any element. If $x\in O_S$, then nothing to prove, if $x \in E$, then by Proposition \ref{prop2.1}, there exists a sequence $\omega = (\omega_1, \omega_2, \ldots) \in S$ such that $x = \Theta(\omega)$. We will construct a sequence $\{x_n\}$ in $O_S$ which converges to $x \in E.$ Fix $c \in C$ and define $x_n = f_{\omega_1 \cdots \omega_n}(c) \in O_S$. Then $x_n \to x \in E$, so $x \in \overline{O}_S$. Thus, $E \cup O_S \subseteq \overline{O}_S$.
\end{proof}
In the following, we construct ISSS sets using the above method and observe other ISSS sets.
 We consider the IFS $\{f_1(x)= (\frac{x}{2}, \frac{y}{2}), f_2= (\frac{x}{2}+ \frac{1}{2}, \frac{y}{2}), f_3= (\frac{x}{2}+ \frac{1}{4}, \frac{y}{2}+ \frac{\sqrt{3}}{4})\}$ with a condensation set $C:=$ a circle as shown in the figures.
 Let $S:= \{1,2\}^{\mathbb{N}}$ and $O_S= \bigcup_{i \in S^*} f_i(C).$ Then, the corresponding ISSS set ($3$rd level) is shown in Figure \ref{Fig1}. We present two examples in Figures~\ref{Fig2} and~\ref{Fig3}, that are ISSS sets but are not of the form $E \cup O_S$; however, in general, such sets need not be compact. In the remainder of the paper, we focus on ISSS sets of the form $E \cup O_S$.
  \begin{figure}[H]
\begin{center}
 \includegraphics[width=8cm, height=6cm]{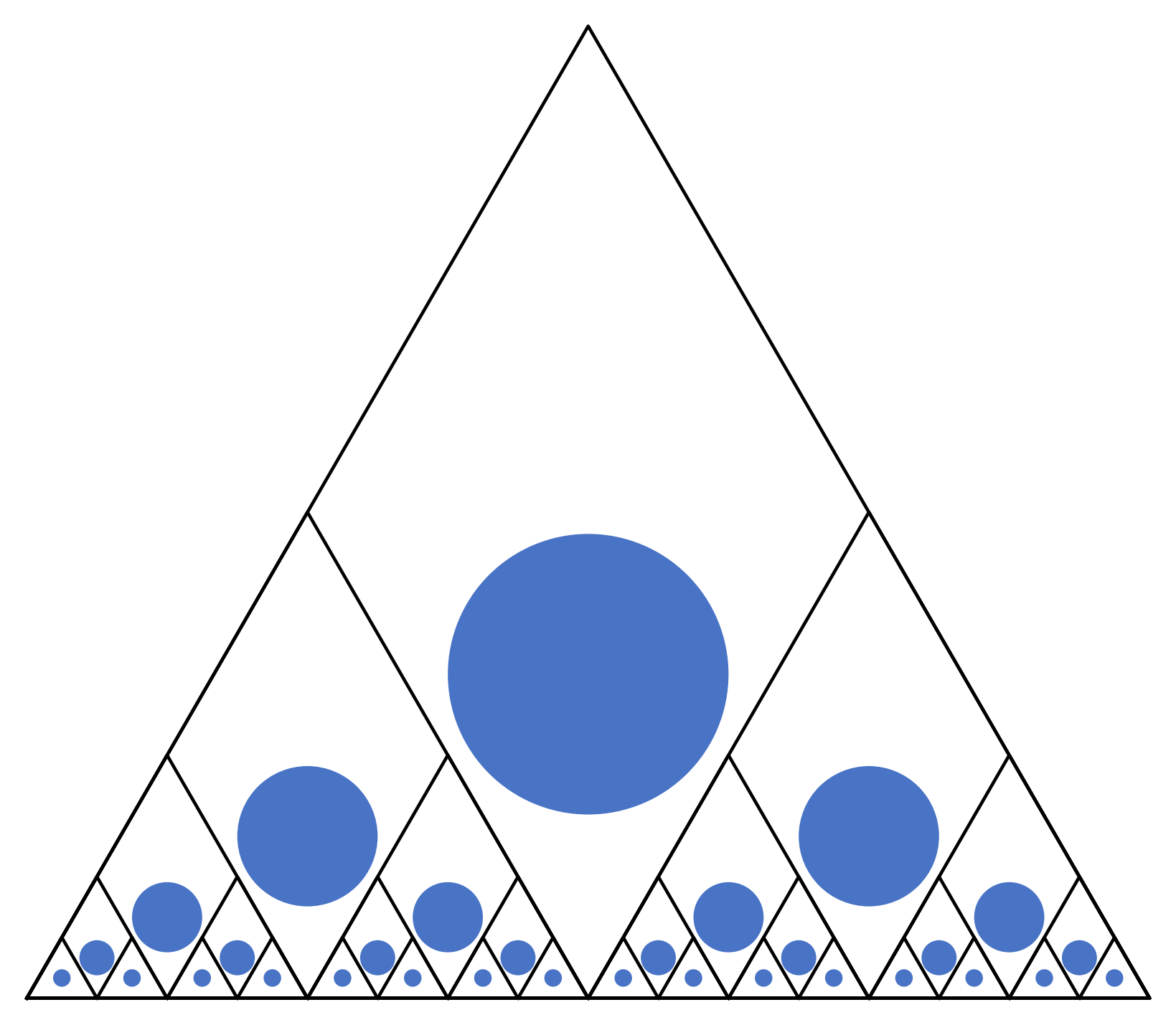}
 \caption{}\label{Fig1}

 \end{center}
 \end{figure}
   \begin{figure}[H]
\begin{center}
 \includegraphics[width=8cm, height=6cm]{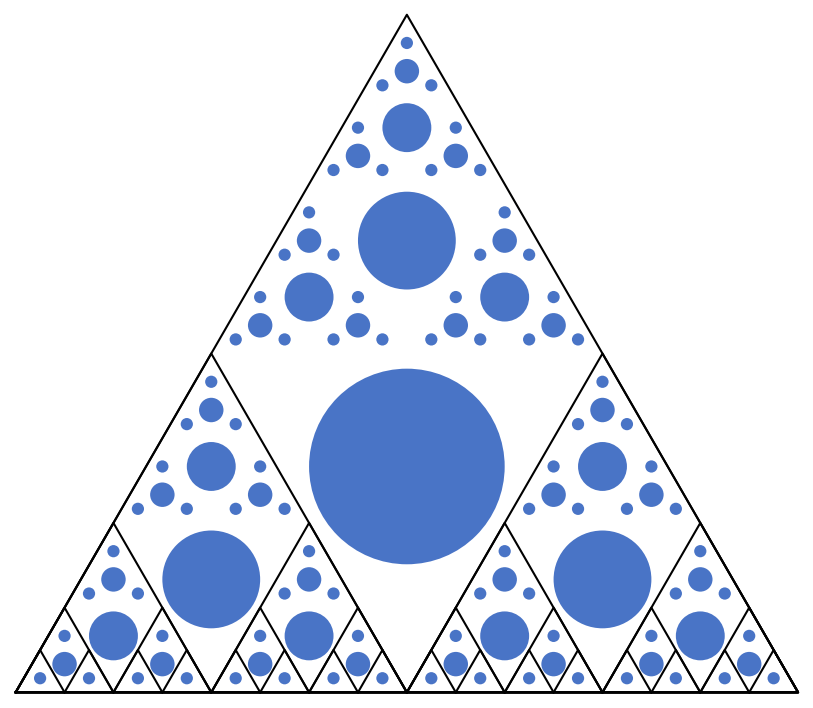}
 \caption{}\label{Fig2}

 \end{center}
 \end{figure}
   \begin{figure}[H]
\begin{center}
 \includegraphics[width=8cm, height=6cm]{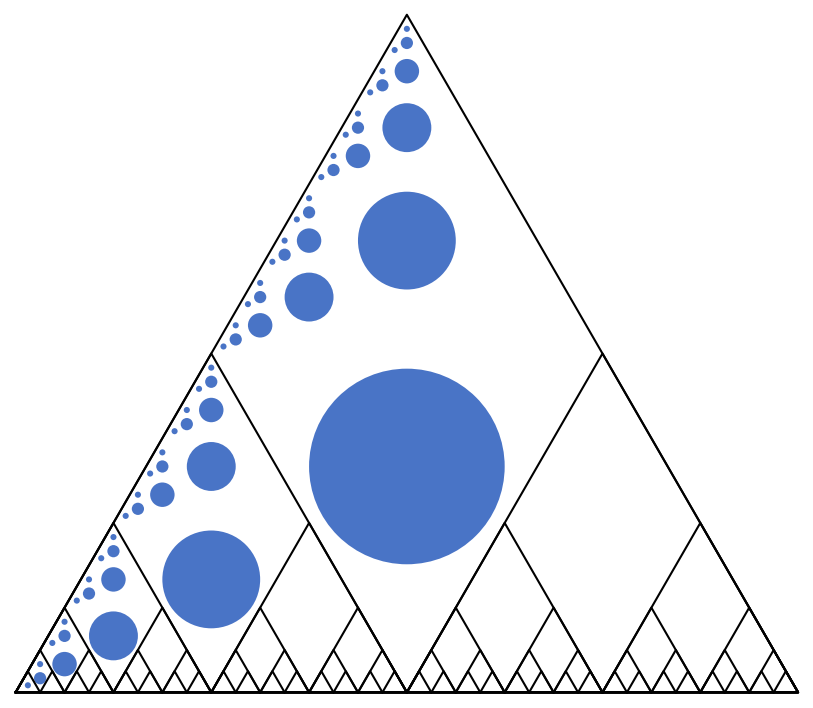}
 \caption{}\label{Fig3}

 \end{center}
 \end{figure}

\section{Dimension of inhomogeneous sub-self-similar sets}
In this section, we aim to estimate the Hausdorff and box dimensions for ISSS sets (obtained from a given SSS set $E$ and a condensation set $C$), which generalize to inhomogeneous self-similar sets. 
\begin{subsection}{Hausdorff dimension}
Let $s$ be the unique non-negative number satisfying $\tau(s) \equiv \lim_{k \to \infty} \bigg(\sum_{i \in S^k} \rho_i^s\bigg)^{1/k}=1$ as defined in Theorem \ref{thm51} and the IFS $\mathcal{I}_C$ satisfies the IOSC, then we have the following corollary.
\begin{corollary}
    Let $ E\cup O_S$ be an ISSS set with SSS set $E$ and condensation set $C.$ Then, $\dim_{H}(E\cup O_S) = \max \{s, \dim_H C\}$.
\end{corollary}
\begin{proof}
    Since the Hausdorff dimension is countably stable, the Hausdorff dimension of $O_S$ equals the Hausdorff dimension of $C$. Additionally, the Hausdorff dimension of the union of $E\cup O_S = \max \{s, \dim_H C\}$.
\end{proof}
\end{subsection}
\begin{subsection}{Upper box dimension}
To estimate the box dimension of these ISSS sets, we will follow Fraser's $\delta$-covering technique. For more details, one can see \cite{Fraser1}. 
\begin{lemma}\label{lema52}
    Let $s$ be a unique non-negative number defined in Theorem \ref{thm51}. Then for any $t > s,$ there exists a constant $m_t$ (depends only on $t$) such that 
    $$\sum_{\omega \in S^*}\rho_{\omega}^t = m_t < \infty.$$
\end{lemma}
\begin{proof}
   The function $\tau(h) = \lim_{k \to \infty} \bigg( \sum_{\omega \in S^k} \rho_{\omega}^h \bigg)^{1/k}$ defined in Theorem \ref{thm51} is strictly decreasing continuous function with $0 \le \tau(h) < \infty$ and $\sum_{\omega \in S^k} \rho_{\omega}^h \ge \tau(h)^k$ for all $k \in \mathbb{N}.$ Moreover, $\sum_{\omega \in S^*}\rho_{\omega}^h=\sum_{k=1}^{\infty} \sum_{\omega \in S^k}\rho_{\omega}^h$ converges if $\tau(h) < 1$ and diverges if $\tau(h) >1.$ Hence, for any $t > s$, we have $\sum_{\omega \in S^*}\rho_{\omega}^t=\sum_{k=1}^{\infty} \sum_{\omega \in S^k}\rho_{\omega}^t= m_t < \infty,$ as $\tau$ is decreasing function and $\tau(s)=1.$ This completes the proof. 
\end{proof}
For $\delta \in (0,1],$ defining $\delta$-stopping, for $S$ as 
$S(\delta) = \{ \omega \in S^* : \rho_{\omega} < \delta \le \rho_{\omega_{-}}\},$ where $\omega_{-}$ is the prefix of $\omega$ with $|\omega_{-}| = n-1,$ if $|\omega|= n.$
 In the following lemma, we estimate a bound for $|S(\delta)|.$
 \begin{lemma}\label{lemma53}
     If $t > s,$ then 
    $|S(\delta)| \le m_t \rho_{\min}^{-t} \delta^{-t}$ for all $\delta \in (0,1],$ where $\rho_{\min} = \min_{i=1}^N \{ \rho_i\}$
 \end{lemma}
 \begin{proof}
     By the above lemma, we have 
     $m_t =\sum_{\omega \in S^*}\rho_{\omega}^t \ge \sum_{\omega \in S(\delta)}\rho_{\omega}^t.$ Furthermore, since $\rho_\omega \ge \rho_{\omega_{-}} \rho_{\min} \ge \delta \rho_{\min}$ for every $\omega \in S(\delta)$, It follows that
     $$m_t \ge \sum_{\omega \in S(\delta)} \delta^t \rho_{\min}^t = |S(\delta)| \delta^t \rho_{\min}^t.$$
     This completes the proof.
 \end{proof}
 \begin{lemma}
Let $t>s$. Then, for all $\delta \in (0,1)$, we have
 $$|\{\omega \in S^* : \delta \le \rho_\omega\}| \le \frac{\log \delta}{\log \rho_{\max} }\delta^{-t} m_t. $$
 \end{lemma}
 \begin{proof}
     Let $\delta \in (0,1)$ and $\omega \in S^*$ such that $\delta \le \rho_\omega.$ Consequently, $\delta \le \rho_{\max}^{|\omega|},$ which yields $|\omega| \le \frac{\log \delta}{\log \rho_{\max}}.$
     Moreover, we have 
     \begin{equation*}
         \begin{aligned}
             m_t \frac{\log \delta}{\log \rho_{\max}} &\ge \sum_{N \in \mathbb{N}, N < \frac{\log \delta}{\log \rho_{\max}}} m_t \\& \ge \sum_{N \in \mathbb{N}, N < \frac{\log \delta}{\log \rho_{\max}}} \sum_{\omega \in S^N} \rho_{\omega}^{t} \\& \ge \sum_{N \in \mathbb{N}, N < \frac{\log \delta}{\log \rho_{\max}}} \sum_{\omega \in S^N, \delta \le \rho_\omega} \delta^t \\& = |\{\omega \in S^* : \delta \le \rho_\omega\}| \delta^t.
         \end{aligned}
     \end{equation*}
    
 \end{proof}
 \begin{lemma}\label{lemma54}
    Let $\delta \in (0,1].$ Then, we have 
    $$\bigcup_{\omega \in S^*, \rho_\omega < \delta } f_\omega (C) \subseteq \bigcup_{\omega \in S(\delta)} f_\omega(X).$$
 \end{lemma}
 \begin{proof}
    For any $x \in \bigcup_{\omega \in S^*, \rho_\omega < \delta } f_\omega (C),$ there exists a string $\omega = \omega_1\ldots \omega_n \in S^*$ and $c \in C$ such that $S_\omega(c)= x$ with $\rho_\omega < \delta.$ If we assume that the contraction ratio for empty string is $1,$ then clearly there exists some prefix $\omega'= \omega_1\ldots \omega_m$ of $\omega$ such that $\omega' \in S(\delta)$ and $x = f_{\omega'}(f_{\omega_{m+1}\ldots \omega_n} (c)) \in f_{\omega'} (X).$ This completes the proof.
 \end{proof}
\begin{theorem}
    Let $E_C= E \cup O_S$ be an ISSS set. Then, we have 
    $$\max \{\overline{\dim}_B E,\overline{\dim}_B C\} \le \overline{\dim}_B E_C \le \max \{s,\overline{\dim}_B C \}.$$
\end{theorem}
\begin{proof}
    Since the upper box dimension is finitely stable and  monotonic, we have
    $$\max \{\overline{\dim}_B E,\overline{\dim}_B C\} \le \overline{\dim}_B E_C \le \max \{\overline{\dim}_B E,\overline{\dim}_B O_S\}.$$
    By Theorem \ref{thm51} $\overline{\dim}_B E \le s,$ it is sufficient to show that $\overline{\dim}_B O_S \le \max\{s, \overline{\dim}_B C\}.$
    Let $\max\{s,\overline{\dim}_B C\} < t,$ then by the definition of the box-dimension, there exists a constant $c_t$ such that $$N_\delta(C) \le c_t \delta^{-t},$$ for any $\delta \in (0,1].$ Consider
    \begin{equation*}
        \begin{aligned}
         N_\delta(O_S) &=  N_\delta \bigg( C \cup \bigcup_{\omega \in S^*} f_\omega(C) \bigg) \\& \le N_\delta (C) + N_{\delta} \bigg( \bigcup_{\omega \in S^*, \rho_\omega \ge \delta} f_\omega (C) \bigg) +N_{\delta} \bigg( \bigcup_{\omega \in S^*, \rho_\omega < \delta} f_\omega (C) \bigg) \\& \le N_\delta (C) + N_{\delta} \bigg( \bigcup_{\omega \in S^*, \rho_\omega \ge \delta} f_\omega (C) \bigg) +N_{\delta} \bigg( \bigcup_{\omega \in S(\delta)} f_\omega (X) \bigg) \\& \le N_\delta (C) +  \sum_{\omega \in S^*, \rho_\omega \ge \delta}  N_{\delta}(f_\omega (C) ) + \sum_{\omega \in S(\delta)} N_{\delta}(f_\omega (X)) \\& \le N_\delta (C) +  \sum_{\omega \in S^*, \rho_\omega \ge \delta}  N_{{\delta}/{\rho_\omega}}(C)  + \sum_{\omega \in S(\delta)} N_{{\delta}/{\rho_\omega}} (X) \\& \le c_t \delta^{-t} + \sum_{\omega \in S^*, \rho_\omega \ge \delta}  c_t({\delta}/{\rho_\omega})^{-t} +\sum_{\omega \in S(\delta)} N_1(X) \\& \le  c_t \delta^{-t} + c_t \delta^{-t} \sum_{\omega \in S^*}  {\rho_\omega}^{t} +|S(\delta)| N_1(X) \\& \le  c_t \delta^{-t} + c_t \delta^{-t} m_t + m_t {\rho}^{-t}_{\min} \delta^{-t} N_1(X) \\& \le  (c_t + c_t m_t + m_t {\rho}^{-t}_{\min}  N_1(X))\delta^{-t}.
        \end{aligned}
    \end{equation*}
    The third inequality follows from Lemma \ref{lemma54}. Also, $X$ is a compact space and $\delta / \rho_\omega > 1$, the third inequality is a consequence of Lemma \ref{lemma53}. This completes the proof.  
\end{proof}
\end{subsection}
\subsection{Lower box dimension}
In this subsection, we aim to estimate an upper bound on the lower box dimension of the ISSS sets. We observe that, if $\overline{\dim}_B C = \underline{\dim}_B C$, then the box dimension of ISSS sets is equal to the maximum of the dimension of the corresponding SSS set and the box dimension of the condensation set. Since the lower box dimension follows the monotonicity property, we have
\[
\max \left\{ \underline{\dim}_B E, \underline{\dim}_B C \right\} \leq \underline{\dim}_B (E \cup O_S) \leq \max \left\{ s, \overline{\dim}_B C \right\}.
\]

Given the stability of the lower box dimension under closure, we have 
\[
\underline{\dim}_B (E \cup O_S) = \underline{\dim}_B \overline{O}_S = \underline{\dim}_B O_S.
\]
Now, we use the concept of the covering regularity index of the condensation set $C$ from \cite{Fraser1} to estimate an upper bound for the lower box dimension of the set $O_S$. 

For $t \geq 0$ and $\delta \in (0,1]$, the $(t,\delta)$ covering regularity exponent of $C$ is defined as
\[
P_{t,\delta} (C) = \sup \left\{ p \in [0,1] : N_{\delta^p} (C) \geq \delta^{-pt} \right\}.
\]
The $t$-covering regularity exponent of $C$ is then given by
\[
P_{t} (C) = \liminf_{\delta \to 0} P_{t,\delta}(C).
\]

\begin{lemma}(see ~\cite[Lemma ~ 2.6]{Fraser1})\label{lem1}
    We have
    \begin{enumerate}
        \item For all \(t, \delta > 0\), \(P_{t, \delta}(C)\), \(P_t(C) \in [0, 1]\).
        \item \(P_t(C)\) is decreasing in \(t\), and if \(t < \underline{\dim}_B C\), then \(P_t(C) = 1\), and if \(t > \overline{\dim}_B C\), then \(P_t(C) = 0\).
        \item For all \(t > \underline{\dim}_B C\), we have \(P_t(C) \le \underline{\dim}_B C < t\).
    \end{enumerate}
\end{lemma}

\begin{lemma}(see,\cite[Lemma~ 3.6]{Fraser1})\label{lem2}
    For \(t \ge 0\), if \(P_t(C) < 1\), then for all \(\epsilon \in (0, 1 - P_t(C))\), there exists \(\delta \in (0, \epsilon)\) such that
    \[
    P_t(C) - \epsilon < P_{t, \delta}(C) < P_t(C) + \epsilon,
    \]
    and for all \(\delta_0 \in [\delta, \delta^{P_t(C)}]\), we have
    \[
    N_{\delta_0}(C) \le \delta_0^{-t}.
    \]
\end{lemma}

By Lemma \ref{lemma53}, for \(t > s\), the cardinality of the \(\delta\)-stopping set \(S\) is given by
\[
|S(\delta)| \le m_t \rho_{\min}^{-t} \delta^{-t}.
\]
Using the two lemmas above, we can estimate the lower bound on the box dimension.

\begin{theorem}
    We have
    \[
    \underline{\dim}_B (E \cup O_S) \le \max \{s, \underline{\dim}_B C\}.
    \]
\end{theorem}

\begin{proof}
    Let \(t > \max\{s, \underline{\dim}_B C\}\). Then, by Lemma \ref{lem1}, we have \(P_t(C) \le \underline{\dim}_B C < t\). Thus, by Lemma \ref{lem2}, for all \(\epsilon \in (0, 1 - P_t(C))\), there exists \(\delta \in (0, \epsilon)\) such that:
   $$ P_t(C) - \epsilon < P_{t, \delta}(C) < P_t(C) + \epsilon.$$
    
    Moreover, for all \(\delta_0 \in [\delta, \delta^{P_t(C)}]\), we have
    \[
    N_{\delta_0} \le \delta_0^{-t}.
    \]
Let $\epsilon \in (0, 1 - P_t(C))$ be fixed and choose $\delta \in (0, \epsilon)$, we have
\[
\begin{aligned}
    N_{\delta}(O_S) &= N_{\delta}\left(C \cup \bigcup_{\omega \in S^*} f_\omega(C)\right) \\
    &\le \sum_{\omega \in S^*, \delta^{1 - P_{t, \delta}(C) - \epsilon} \le \rho_\omega < 1} N_{\delta} f_\omega(C) + \sum_{\omega \in S^*, \delta \le \rho_\omega < \delta^{1 - P_{t, \delta}(C) - \epsilon}} N_{\delta} f_\omega(C) \\
    &\quad \hspace{4.5cm}+ N_{\delta} \left(\bigcup_{\omega \in S^*, \rho_\omega < \delta} f_\omega(C)\right) + N_{\delta}(C) \\
    &\le \sum_{\omega \in S^*, \delta^{1 - P_{t, \delta}(C) - \epsilon} \le \rho_\omega < 1} N_{\delta / \rho_\omega}(C) + \sum_{\omega \in S^*, \delta \le \rho_\omega < \delta^{1 - P_{t, \delta}(C) - \epsilon}} N_{\delta / \rho_\omega}(C) \\
    &\quad \hspace{4.4cm}+ N_{\delta} \left( \bigcup_{\omega \in S(\delta), \rho_\omega < \delta} f_\omega(X) \right) + N_{\delta}(C) \\
    &\le \sum_{\omega \in S^*, \delta^{1 - P_{t, \delta}(C) - \epsilon} \le \rho_\omega < 1} (\delta / \rho_\omega)^{-t} + \sum_{\omega \in S^*, \delta \le \rho_\omega < \delta^{1 - P_{t, \delta}(C) - \epsilon}} N_{\delta^{P_{t, \delta}(C) + \epsilon}}(C) \\
    &\quad \hspace{4.5cm}+ \sum_{\omega \in S(\delta)} N_{\delta / \rho_\omega}(X) + \delta^{-t} \\
    &\le \sum_{\omega \in S^*, \delta^{1 - P_{t, \delta}(C) - \epsilon} \le \rho_\omega < 1} (\delta / \rho_\omega)^{-t} + \sum_{\omega \in S^*, \delta \le \rho_\omega < \delta^{1 - P_{t, \delta}(C) - \epsilon}} N_{\delta / \rho_\omega}(C) \\
    &\quad \hspace{4.5cm}+ N_1(X) |S(\delta)| + \delta^{-t} \\
    &\le \sum_{\omega \in S^*} (\delta / \rho_\omega)^{-t} + |\{\omega \in S^* : \delta \le \rho_\omega\}| \delta^{-(P_{t, \delta}(C) + \epsilon)t} \\
    &\quad \hspace{4.5cm}+ N_1(X) m_t \rho_{\min}^{-t} \delta^{-t} + \delta^{-t} \\
    &\le \delta^{-t} \sum_{\omega \in S^*} \rho_\omega^t + |\{\omega \in S^* : \delta \le \rho_\omega\}| \delta^{-(P_{t, \delta}(C) + \epsilon)t} \\
     \end{aligned}\]
    \[ \begin{aligned}
    &\quad \hspace{4.5cm}+ N_1(X) m_t \rho_{\min}^{-t} \delta^{-t} + \delta^{-t} \\
    &\le \left( m_t + N_1(X) m_t \rho_{\min}^{-t} + 1 + m_t \frac{\log \delta}{\log \rho_{\max}} \delta^{-(P_{t, \delta}(C) + \epsilon)t} \right) \delta^{-t} \\
    &\le \left( m_t + N_1(X) m_t \rho_{\min}^{-t} + 1 + m_t \frac{\log \delta}{\log \rho_{\max}} \delta^{-(P_{t}(C) + 2\epsilon)t} \right) \delta^{-t}.
\end{aligned}
\]
Letting \(\epsilon \to 0\), we get \(\underline{\dim}_B O_S \le t\). This concludes \(\underline{\dim}_B (E \cup O_S) \le \max\{s, \underline{\dim}_B C\}\).

\end{proof}
\begin{remark}
Notice that if we take $S = I^{\infty}$ (the code space), then the given sub-self-similar set reduces to the self-similar set associated with the underlying IFS. In this case, the above dimension estimation yields the Hausdorff and box dimensions of the inhomogeneous self-similar set associated with the condensation set $C$.
\end{remark}

\section{Continuity of the Hausdorff dimension of ISSS}
In this section, we show that for each ISSS set of the form  $F:= E \cup O_S$, where $E$ is an SSS set and $C$ is the condensation set, there exists a sequence of inhomogeneous IFSs $\{I_k\}$ with corresponding ISSS sets $\{F_k\}$ such that $F = \bigcup_{k \in \mathbb{N}} F_k$. Furthermore, we discuss the continuity of the Hausdorff dimension of ISSS sets.
 
\begin{theorem} \label{thm000}
Let $F = E \cup O_S$ be an ISSS set corresponding to an IFS $\{f_i\}_{i=1}^N$ with condensation set $C$, that satisfies the open set condition. Then, there exists a sequence of IFSs $\{I_k\}_{k=1}^{\infty}$ such that
\[
F = \bigcup_{k=1}^{\infty} F_k,
\]
where $F_k = E_k \cup O_S$ are the ISSS sets corresponding to the IFSs $I_k$, each with condensation set $C$.
\end{theorem}

\begin{proof}
Let $S^k :=\{\omega = \omega_1 \omega_2\ldots \omega_k : \exists \omega_1 \ldots \omega_n\ldots \in S\}$, and
let $I_k := \{ f_{\omega} : \omega \in S^k \}$ be IFS with corresponding ISSS set $F_k$ for each $k \in \mathbb{N}$. Then, for any $ x \in F = E \cup O_S$, we consider the following two cases:

\textbf{Case 1:} If $ x \in E$, then there exists a string $ \omega = \omega_1 \omega_2 \ldots \in S$ such that
\[
x = \lim_{n \to \infty} f_{\omega_1} \circ f_{\omega_2} \circ \ldots \circ f_{\omega_n} (z).
\]

For each $k \in \mathbb{N}$, we rewrite $\omega$ as $\omega:= (\omega_1 \omega_2 \ldots \omega_k)|_{0}(\omega_{k+1}\omega_{k+2}\ldots \omega_{2k})|_{1}\ldots (\omega_{jk+1}\omega_{jk+2}\ldots \\ \omega_{(j+1)k})|_{j}\ldots\in S.$ Since, $S$ is closed under left shift, it follows that for each $j= 0,1,\ldots$ each block $\omega_j:=\omega_{jk+1} \ldots \omega_{(j+1)k} \in S^k$. Consequently, we obtain
\[
x = \lim_{j \to \infty} f_{\omega_0} \circ f_{\omega_1} \circ \ldots \circ f_{\omega_j} (z),
\]
and hence $x \in \bigcup_{k=1}^{\infty} F_k$.

\textbf{Case 2:} If $x \in O_S$, then there exists some $\omega = \omega_1 \omega_2 \ldots \omega_n \in S^*$ and $c \in C$ such that
\[
f_{\omega_1} \circ \ldots \circ f_{\omega_n}(c) = x.
\]
Clearly, $\omega_1 \ldots \omega_n \in S^n$. Hence, we have $x \in \bigcup_{k=1}^{\infty} F_k \cup O_S$.

Conversely, suppose $x \in \bigcup_{k=1}^{\infty} F_k \cup O_S$. If $x\in O_S$, then the claim holds trivially. Now, suppose $x \in F_k$ for some $k$, then there exists a sequence $\{\omega_j\}_{j=1}^{\infty}$ in $S^k$ such that
\[
x = \lim_{j \to \infty} f_{\omega_1} \circ f_{\omega_2} \circ \ldots \circ f_{\omega_j} (z).
\]
Using the fact that $\Theta(S) = E$ and that the concatenation $\omega = \omega_1 \omega_2 \ldots \omega_j\ldots \in S$, we conclude that $x \in F$.

Thus, we obtain $F = \bigcup_{k=1}^{\infty} F_k$.
\end{proof}
\begin{note}
    It is worth noting that, if the underlying IFS satisfies the OSC, then the IFS $I_k$ also satisfy the OSC for each $k\in \mathbb{N},$ and $\{E_k\}_{k=1}^{\infty}$ is a monotonic increasing sequence. If $x \in E_k,$ then there exists a sequence $\{\omega_j\}_{j=1}^{\infty}$ in $S^{k}$ such that $\lim_{j \to \infty} f_{\omega_1}\circ f_{\omega_2}\circ \cdots \circ f_{\omega_j}(z)=x.$
    By the fact that $x \in F$ and $\Theta(S) = E,$  the concatenation $\Omega :=\omega_1\ldots \omega_j\ldots \in S.$ Since, $\omega_j \in S^{k},$ by our construction there exist $\omega^j \in S$ such that $\omega^j|_{k}= \omega_j := \omega^j_1\ldots \omega^{j}_k$ for each $j\in \mathbb{N}.$ Now, we rewrite $\Omega=(\omega^1_1\ldots\omega^1_{k} \omega^2_{1})(\omega^2_2 \ldots \omega^2_{k}\omega^3_{1}\omega^3_{2})\ldots$, then clearly each block is of length $k+1$ and an element of $S^{k+1}.$ This implies that $x \in E_{k+1}.$
\end{note}
\begin{lemma}{\cite{Liu}}\label{prop90}
    Let $S \subseteq l^{\infty}$ be any shift-invariant set in code space  $I^{\infty}.$ Then 
 $$\dim_{H} S = \dim_B S = \lim_{k \to \infty} s_k,$$
 where the relation $\sum_{\omega \in S^k} (\rho_\omega)^{s_k}=1$ determines $s_k$ uniquely.
 
\end{lemma}
\begin{corollary}
   Let $F= E \cup O_S$ be the ISSS set corresponding to the IFS $\{f_i\}_{i=1}^N.$ Then, there exists a sequence of IFSs $\{I_k\}_{k=1}^{\infty}$ with ISSS sets $\{F_k\}$ and under IOSC $\dim_H E= \lim_{k \to \infty} \dim_H E_k = \lim_{k \to \infty} s_k,$ where $s_k$ is the Hausdorff dimension of $F_k$ defined in Lemma \ref{prop90}.
\end{corollary}
\begin{proof}
By Theorem \ref{thm000}, there exists a sequence of IFSs $\{I_k\}_{k=1}^{\infty}$ with ISSS sets $F_k$ such that $F= \bigcup_{k=1}^\infty F_k.$
 It is easy to observe that if \( \{ f_i \}_{i=1}^N \) satisfies the IOSC, then each \( I_k \) also satisfies the IOSC. 
 If  $s_k$ represents the Hausdorff dimension of $E_k$, which satisfies

\[
\lim_{k \to \infty} \left( \sum_{\omega \in S^k} \rho_\omega^{s_k} \right)^{1/k} = 1,
\]

then the Hausdorff dimension of $F_k$ is $\max\{s_k, \dim_H(C)\}$.

By Lemma \ref{prop90}, we have
\[
\dim_H E = \lim_{k \to \infty} \dim_H E_k = \lim_{k \to \infty} s_k,
\]
and $
\dim_H F_k \to \dim_H F
.$ This completes the assertion. 
\end{proof}

\section{Self-similar structure on the product of IFS}
In this section, we define the product of inhomogeneous iterated function systems (IFSs) and examine the associated separation conditions. We conclude by demonstrating that the product of the inhomogeneous invariant measures corresponding to the individual IFSs constitutes an inhomogeneous type invariant measure for the product IFS.
Let $(X,d_1)$ and $(Y,d_2)$ be two complete metric spaces; then it is well known that the product space $(X \times Y,d)$ is a complete metric space with respect to the metric $d$ given by
$$d((x,y),(z,w))= \max\{d_1(x,z),d_2(y,w)\}.$$
Now, for two IFSs
$\mathcal{I}_1= \{(X,d_1); f_1,f_2,\ldots,f_N\}$ and $\mathcal{I}_2= \{(Y,d_2); g_1,g_2,\ldots,g_M\}$ we define product IFS $\mathcal{P}= \{X \times Y; \Phi_{ij}, (i,j) \in \{1,2,\ldots,N\}\times \{1,2,\ldots,M\}\}$, where $\Phi_{(ij)}(x,y) = (f_i(x),g_j(y)).$ 
It is easy to observe that $\Phi_{ij}$ is contraction for all $ij \in \{1,2,\ldots,N\}\times \{1,2,\ldots,M\}.$
For any $(x,y),(z,w) \in X \times Y,$ we have 
\begin{equation*}
    \begin{aligned}
        d(\Phi_{ij}(x,y),\Phi_{ij}(z,w)) &= d((f_i(x),g_j(y),(f_i(z),g_j(w))) \\& = \max(d_1(f_i(x),f_i(z)),d_2(g_j(y),g_j(w))) \\& \le \{c^1_i d_1(x,z), c^2_j d_2(z,w)\} \\& \le c \max\{d_1(x,z), d_2(z,w)\} \\&= c~d((x,y),(z,w)),
    \end{aligned}
\end{equation*}
where $c= \max\{c^1_i,c^2_j\}.$ Let $A_{\mathcal{I}_1}$ and $A_{\mathcal{I}_2}$ be the attractors corresponding to the IFSs ${\mathcal{I}_1}$ and ${\mathcal{I}_2},$ respectively then
 It is easy to verify that the attractor corresponding to the product IFS $\mathcal{P}$ is $A_{\mathcal{I}_1} \times A_{\mathcal{I}_2}.$
 We claim that $\bigcup_{i,j} f_i(A_{\mathcal{I}_1}) \times g_j(A_{\mathcal{I}_2}) = \bigcup_{i} f_i(A) \times \bigcup_{j} g_j(A_{\mathcal{I}_2}).$
 
  \begin{equation*}
     \begin{aligned}
       \text{Let}~ &(x,y) \in \bigcup_{i,j} f_i(A_{\mathcal{I}_1}) \times g_j(A_{\mathcal{I}_2}) \\& \Longleftrightarrow (x,y) \in f_i(A_{\mathcal{I}_1}) \times g_j(A_{\mathcal{I}_2}) ~ \text{for some}~ (ij) \in \{1,2,\ldots,N\}\times \{1,2,\ldots,M\} \\& \Longleftrightarrow (x,y) \in f_i(A_{\mathcal{I}_1}) \times g_j(A_{\mathcal{I}_2}) \\& \Longleftrightarrow x \in f_i(A_{\mathcal{I}_1}) ~\text{and}~ y \in  g_j(A_{\mathcal{I}_2}) \\& \Longleftrightarrow x \in \bigcup_{i} f_i(A_{\mathcal{I}_1}) ~ \text{and} ~ y \in \bigcup_{j} g_j(A_{\mathcal{I}_2}) \\& \Longleftrightarrow (x,y) \in \bigcup_{i} f_i(A_{\mathcal{I}_1}) \times \bigcup_{j} g_j(A_{\mathcal{I}_2}).
     \end{aligned}
 \end{equation*}
Now,
 \begin{equation*}
     \begin{aligned}
   A_{\mathcal{I}_1} \times A_{\mathcal{I}_2} &= \bigcup_{i} f_i(A_{\mathcal{I}_1}) \times \bigcup_{i} g_j(A_{\mathcal{I}_1}) \\&= \bigcup_{i,j} f_i(A_{\mathcal{I}_1}) \times g_j(A_{\mathcal{I}_2}).
     \end{aligned}
 \end{equation*}
 Similarly, we define the product of two inhomogeneous IFS. Let $\mathcal{I}_{C_1}=   \{(X,d_1); f_1,f_2,\ldots,\\f_N,C_1\}$ and $\mathcal{I}_{C_2}= \{(Y,d_2); g_1,g_2,\ldots,g_M, C_2\}$ be two inhomogeneous IFSs with condensation set $C_1 \subset X$ and $C_2 \subset Y$ respectively, then we define inhomogeneous product IFS $\mathcal{P}_C= \{X\times Y; \Phi_{ij},(C_1,C_2),(i,j) \in \{1,2,\ldots,N\}\times \{1,2,\ldots,M\}\}.$ If $A_{C_1}$ and $A_{C_2}$ are the inhomogeneous attractors corresponding to the IFSs $\mathcal{I}_{C_1}$ and $\mathcal{I}_{C_2},$ respectively.  $A_{C_1} \times A_{C_2}$ is a super-self-similar set corresponding to the product IFS $\mathcal{P}_C,$ i.e., $A_{C_1} \times A_{C_2} \supseteq \bigcup_{i,j} \Phi_{ij}(A_{C_1}\times A_{C_2}) \cup C_1 \times C_2.$
\begin{theorem}
    Let $A_{C_1}$ and $A_{C_2}$ be the inhomogeneous attractors corresponding to the inhomogeneous IFSs $\mathcal{I}_{C_1}=    \{(X,d_1); f_1,f_2,\ldots,f_N,C_1\}$ and $\mathcal{I}_{C_2}= \{(Y,d_2); g_1,g_2,\ldots,g_M, C_2\},$ respectively. Let $\mathcal{P}_C= \{X\times Y; \Phi_{ij},(C_1,C_2),(i,j) \in \{1,2,\ldots,N\}\times \{1,2,\ldots,M\}\}$ be the product IFS, then $A_{C_1} \times A_{C_2}$ be a super-self-similar set.
\end{theorem}
\begin{proof}
   Since $A_{C_1}$ and $A_{C_2}$ are the inhomogeneous attractors. We have
   $$A_{C_1}= \bigcup_{i=1}^N f_i (A_{c_1}) \cup C_1, ~ \text{and}~ A_{C_2}= \bigcup_{j=1}^M g_j(A_{C_2}) \cup C_2.$$
   Also,
   \begin{equation*}
   \begin{aligned}
       A_{C_1} \times A_{C_2} &= \bigg(\bigcup_{i} f_i(A_{C_1}) \cup C_1 \bigg) \times \bigg(\bigcup_{j} g_j(A_{C_2}) \cup C_2 \bigg) \\&= \bigg(\bigcup_{i} f_i(A_{C_1}) \times \bigcup_{j} g_j(A_{C_2})\bigg) \cup \bigg(\bigcup_{i} f_i(A_{C_2}) \times C_2 \bigg) \cup \\& \hspace{4cm} \bigg(C_1 \times \bigcup_{j} g_j(A_{C_2}) \bigg) \cup (C_1 \times C_2)  \\& \supseteq \bigcup_{i,j} f_i(A_{C_1}) \times g_j(A_{C_2}) \cup (C_1 \times C_2).
       \end{aligned}
   \end{equation*}
\end{proof}
 In the next theorem, we see that the inhomogeneous separation conditions on $\mathcal{I}_{C_1}$ and $\mathcal{I}_{C_2}$ can be transferred to the inhomogeneous product IFS $\mathcal{P}_C$.
\begin{theorem}
    Let $\mathcal{I}_1 = \{(X, d_1); f_1, f_2, \ldots, f_N, C_1\}$ and $\mathcal{I}_2 = \{(Y, d_2); g_1, g_2, \ldots, g_M, C_2\}$ be two IFSs. Then the following hold:
    \begin{enumerate}
        \item The IFS $\mathcal{P}_C$ satisfies the ISSC if the IFSs $\mathcal{I}_{C_1}$ and $\mathcal{I}_{C_2}$ satisfy the ISSC.
        \item The IFS $\mathcal{P}_C$ satisfies the IOSC if the IFSs $\mathcal{I}_{C_1}$ and $\mathcal{I}_{C_2}$ satisfy the IOSC.
    \end{enumerate}
\end{theorem}

\begin{proof}
    \begin{enumerate}
        \item Since $\mathcal{I}_{C_1}$ and $\mathcal{I}_{C_2}$ satisfy the ISSC, we have 
        \[
        f_i(A_{C_1}) \cap f_j(A_{C_1}) = \emptyset \quad \text{for all } i \neq j, \qquad f_i(A_{C_1}) \cap C_1 = \emptyset \quad \text{for all } 1 \le i \le N,
        \]
        and
        \[
        g_k(A_{C_2}) \cap g_l(A_{C_2}) = \emptyset \quad \text{for all } k \neq l, \qquad g_k(A_{C_2}) \cap C_2 = \emptyset \quad \text{for all } 1 \le k \le M.
        \]
        This leads to
        \[
        \Phi_{ij}(A_{C_1} \times A_{C_2}) \cap \Phi_{kl}(A_{C_1} \times A_{C_2}) = (f_i(A_{C_1}) \times g_j(A_{C_2})) \cap (f_k(A_{C_1}) \times g_l(A_{C_2})) = \emptyset,
        \]
        for all $(i,j) \neq (k,l)$. Since $A_{\mathcal{P}_C} \subseteq A_{C_1} \times A_{C_2}$, it follows that
        \[
        \Phi_{ij}(A_{\mathcal{P}_C}) \cap \Phi_{kl}(A_{\mathcal{P}_C}) = \emptyset \quad \text{for all } (i,j) \neq (k,l).
        \]
        Furthermore,
        \[
        \Phi_{ij}(A_{C_1} \times A_{C_2}) \cap (C_1 \times C_2) = (f_i(A_{C_1}) \times g_j(A_{C_2})) \cap (C_1 \times C_2) = \emptyset.
        \]
        Therefore, the product IFS $\mathcal{P}_{C}$ satisfies the ISSC.

        \item Since $\mathcal{I}_{C_1}$ and $\mathcal{I}_{C_2}$ satisfy the IOSC, there exist open sets $U \subset X$ and $V \subset Y$ such that
        \[
        \bigcup_{i=1}^N f_i(U) \subset U, \quad f_i(U) \cap f_j(U) = \emptyset \quad \text{for all } i \neq j, \quad \text{and } C_1 \subset \overline{U},
        \]
        and
        \[
        \bigcup_{k=1}^M g_k(V) \subset V, \quad g_k(V) \cap g_l(V) = \emptyset \quad \text{for all } k \neq l, \quad \text{and } C_2 \subset \overline{V}.
        \]
        This implies
        \[
        \Phi_{ij}(U \times V) \cap \Phi_{kl}(U \times V) = (f_i(U) \times g_j(V)) \cap (f_k(U) \times g_l(V)) = \emptyset \quad \text{for all } (i,j) \neq (k,l).
        \]
        Moreover,
        \[
        \Phi_{ij}(U \times V) = f_i(U) \times g_j(V) \subset U \times V \quad \text{for all } (i,j) \in \{1, \ldots, N\} \times \{1, \ldots, M\}.
        \]
        Also, since $C_1 \subset \overline{U}$ and $C_2 \subset \overline{V}$, we have
        \[
        C_1 \times C_2 \subset \overline{U} \times \overline{V}.
        \]
        Define an open set $O:= U \times V \subset X \times Y$. Then
        \[
        \bigcup_{ij} \Phi_{ij}(O) \subset O, \quad \Phi_{ij}(O) \cap \Phi_{kl}(O) = \emptyset \quad \text{for all } (i,j) \neq (k,l), \quad \text{and } C \subset \overline{O}.
        \]
        Hence, the product IFS $\mathcal{P}_C$ satisfies the IOSC.
    \end{enumerate}
\end{proof}
Let us see evidence of a connection between fractal measures of the IFSs $\mathcal{I}_1$, $\mathcal{I}_2$, and  $\mathcal{P}_C.$ 
Consider the probability vectors $(p_0, p_1, \ldots, p_N)$ and $(q_0, q_1, \ldots, q_M)$ associated with the IFSs $\mathcal{I}_1$ and $\mathcal{I}_2$, respectively.
Let $\nu_1$ and $\nu_2$ be Borel probability measures on $X$ and $Y$ with supports $C_1$ and $C_2$, respectively. Denote the inhomogeneous invariant measures associated with the IFSs $\mathcal{I}_1$ and $\mathcal{I}_2$ by $\mu_{\mathcal{I}_1}$ and $\mu_{\mathcal{I}_2}$, respectively. Then we have the following result.
\begin{theorem}
    Let $\mathcal{P}_C$ be the inhomogeneous product IFS of $\mathcal{I}_1$ and $\mathcal{I}_2$ as considered above, and let
    \[
\bigg((r_{ij})_{\substack{1\le i\le N,\\ 1\le j\le M}}, \sum_{i=1}^N q_0 p_i+ \sum_{j=1}^M p_0 q_j+ p_0 q_0\bigg) ~\text{where}~ r_{ij} = p_i q_j
\]
be a probability vector. Then, for some Borel probability measure $\nu$ the product measure $\mu_{\mathcal{I}_1} \times \mu_{\mathcal{I}_2}$ is an inhomogeneous-type measure, given by
\[
\mu_{\mathcal{I}_1} \times \mu_{\mathcal{I}_2} = \sum_{i=1,j=1}^{N,M} r_{ij} \mu_{\mathcal{P}_C} \circ \Phi_{ij}^{-1} + \left( \sum_{i=1}^N q_0 p_i + \sum_{j=1}^M p_0 q_j + p_0 q_0 \right) \nu.
\] 
\end{theorem}

\begin{proof}
Recall that
\[
\mu_{\mathcal{I}_1} = \sum_{i=1}^{N} p_i \mu_{\mathcal{I}_1} \circ f_i^{-1} + p_0 \nu_1 ~\text{and}~\mu_{\mathcal{I}_2} = \sum_{j=1}^{M} q_j \mu_{\mathcal{I}_2} \circ g_j^{-1} + q_0 \nu_2.
\]
It follows that
\begin{align*}
\mu_{\mathcal{I}_1} \times \mu_{\mathcal{I}_2} &= \left( \sum_{i=1}^{N} p_i \mu_{\mathcal{I}_1} \circ f_i^{-1} + p_0 \nu_1 \right) \times \left( \sum_{j=1}^{M} q_j \mu_{\mathcal{I}_2} \circ g_j^{-1} + q_0 \nu_2 \right) \\
&= \sum_{i=1,j=1}^{N,M} r_{ij} \mu_{\mathcal{I}_1} \circ f_i^{-1} \times \mu_{\mathcal{I}_2} \circ g_j^{-1} \\
&\quad + \sum_{i=1}^{N} p_i q_0 \mu_{\mathcal{I}_1} \circ f_i^{-1} \times \nu_2 + \sum_{j=1}^{M} p_0 q_j \nu_1 \times \mu_{\mathcal{I}_2} \circ g_j^{-1} + p_0 q_0 \nu_1 \times \nu_2 \\
&= \sum_{i=1,j=1}^{N,M} r_{ij} \left( \mu_{\mathcal{I}_1} \times \mu_{\mathcal{I}_2} \right) \circ \Phi_{ij}^{-1} \\
&\quad + \left( \sum_{i=1}^N p_i q_0 + \sum_{j=1}^M p_0 q_j + p_0 q_0 \right) \nu,
\end{align*}
where
\[
\nu := \frac{
\sum_{i=1}^{N} p_i q_0 \mu_{\mathcal{I}_1} \circ f_i^{-1} \times \nu_2 + \sum_{j=1}^{M} p_0 q_j \nu_1 \times \mu_{\mathcal{I}_2} \circ g_j^{-1} + p_0 q_0 \nu_1 \times \nu_2
}{
\sum_{i=1}^N p_i q_0 + \sum_{j=1}^M p_0 q_j + p_0 q_0
}
\]
is a Borel probability measure on $X \times Y$.
\end{proof}
\section{Conclusion and future remarks}\label{section 8}
We have established significant results in dimension theory and constructed new fractal sets using Falconer's sub-self-similar concept. This work generalizes several existing results in the literature and, in the future, will be extended to more general fractal sets or attractors of dynamical systems. The fractal sets we have constructed have potential applications in the animation industry, computer-aided geometric design (CAGD), and graphic design. In future work, we plan to estimate the Assouad dimension of these ISSS sets and investigate the dimension spectrum.

\section*{Statements and Declarations}
\textbf{Data availability:} Data sharing is not applicable to this article as no data sets were generated or analyzed during the current study.\\
  \textbf{Funding:} The first author thanks IIIT Allahabad (Ministry of Education, India) for financial support through a Senior Research Fellowship. The second author acknowledges financial support from the SEED grant project of IIIT Allahabad.\\
 \textbf{Conflict of interest:} 
We declare that we have no conflicts.\\
 \textbf{Author Contributions:} 
 All authors contributed equally to this manuscript.
 \bibliographystyle{amsplain}

\begin{thebibliography}{999}
\bibitem{Ach} R. Achour, B. Selmi, General fractal dimensions of typical sets and measures, Fuzzy Sets and Systems, 490 (2024) 109039.
\bibitem{Acho} R. Achour, Z. Li, B. Selmi, and T. Wang, General fractal dimensions of graphs of products and sums of continuous functions and their decompositions, Journal of Mathematical Analysis and Applications, 538(2) (2024) 128400.
\bibitem{AV1} E. Agrawal, S. Verma, Dimension preserving approximation and estimation: fractal surfaces and Riemann–Liouville fractional integrals, Contemporary Mathematics (AMS), 825 (2025) 1-24.
\bibitem{AV2} E. Agrawal, S. Verma, Dimensions and stability of invariant measures supported on fractal surfaces,
Discrete and Continuous Dynamical Systems-S, 22 (2026) 155-181.
\bibitem{B} S. Baker, J. M. Fraser, and Andr\'{a}s M\'{a}th\'{e}, Inhomogeneous self-similar sets with overlaps, Ergodic Theory and Dynamical Systems, 39(1)
(2019) 1-18.
\bibitem{MF1} M. F. Barnsley, Fractal functions and interpolation, Constr. Approx., 2 (1986) 303–329.
\bibitem{MF3} M. F. Barnsley, S. Demko, Iterated function systems and the global construction of fractals, Proc. Roy. Soc. London Ser. A, 399 (1985) 1817-1824.

\bibitem{MF2} M. F. Barnsley, Fractals Everywhere, Academic Press, Orlando, Florida, 1988.

\bibitem{bFalc} G. C. Boore, and K. J. Falconer, Attractors of directed graph IFSs that are not standard IFS attractors and their Hausdorff measure, Mathematical Proceedings of the Cambridge Philosophical Society, 154(2) (2013) 325-349.
\bibitem{Bu} S. A. Burrell, J. M. Fraser, The dimensions of inhomogeneous self-affine sets, Annales Academiæ Scientiarum Fennicæ Mathematica, 45 (2020) 313–324.
\bibitem{Celik} D. Celik, S. Kocak, Y. \"Ozdemir, Fractal interpolation on the Sierpi\'nski Gasket, J. Math. Anal. Appl., 337 (2008) 343-347.
\bibitem{SS5} S. Chandra, and S. Abbas, On fractal dimensions of fractal functions using function spaces, Bulletin of the Australian Mathematical Society, 106(3) (2022) 470-480.
\bibitem{Chana} S. Chandra, and S. Abbas, Analysis of fractal dimension of mixed Riemann-Liouville integral, Numerical Algorithms, 91(3) (2022) 1021-1046.
\bibitem{Cui} M. Cui, Z. Li, and B. Selmi, A Variational Principle of the Packing Topological Pressure on Subsets for Non-autonomous Iterated Function Systems, Journal of Dynamical and Control Systems, 32(1) (2026) 2.
\bibitem{DV1} S. Dubey, S. Verma, Fractal dimension for a class of inhomogeneous graph-directed attractors, Discrete and Continuous Dynamical Systems-S, 22 (2026) 139-154. 
\bibitem{DV2} S. Dubey, S. Verma, Inhomogeneous graph-directed attractors and fractal measures, The Journal of Analysis, 32(1) (2024) 157-170.
\bibitem {Fal} K. J. Falconer, Fractal Geometry: Mathematical Foundations and Applications, John Wiley Sons Inc., New York, 1999.

\bibitem{Fraser} K. J. Falconer, J. M. Fraser, The horizon problem for prevalent surfaces, Math. Proc. Camb. Phil. Soc., 151 (2011) 355.

\bibitem{Fal1}  K. J. Falconer, Sub-Self-Similar Sets, Trans. Amer. Math. Soc., 347 (1995) 3121-3129.

\bibitem{Fraser1} J. M. Fraser, Inhomogeneous self-similar sets and box dimensions, Studia Math., 213 (2012) 133-156.
\bibitem{Frase} J. M. Fraser, Inhomogeneous self-affine carpets, Indiana University Mathematics Journal, (2016) 1547-1566.
\bibitem{11jordan} R. A. Gordon, Real Analysis: A First Course, 2nd edition, Boston, Pearson Education Inc.,
2002. 

\bibitem{GCV1} Gurubachan, V. Chandramouli, S. Verma, Analysis of $\alpha$-fractal functions without boundary point conditions on the Sierpi\'nski gasket, Applied Mathematics and Computation, 486 (2025) 129072
\bibitem{GCV2} Gurubachan, V. Chandramouli, S. Verma, Fractal Dimension of $\alpha$-Fractal Functions Without Endpoint Conditions, Mediterranean Journal of Mathematics, 21(3) (2024) 71.
\bibitem{MHOCHMAN} M. Hochman, On self-similar sets with overlaps and inverse theorems for entropy, Annals of Mathematics 180(2) (2014) 773-822.	
\bibitem{Hut} J. E. Hutchinson, Fractals and self-similarity, Indiana Uni. Math. J., 30(5) (1981) 713-747.


\bibitem{Verma21} S. Jha, S. Verma, Dimensional Analysis of $\alpha$-Fractal Functions, Results Math, 76(4) (2021) 186.
\bibitem{K} A. K\"{a}enm\"{a}ki, and J. Lehrb\"{a}ck, Measures with predetermined regularity and inhomogeneous self-similar sets, (2017) 165-184.
 \bibitem {Kig} J. Kigami, Analysis on Fractals, Cambridge University Press, Cambridge (2001).

\bibitem{KVB1} A. Kumar, S. K. Verma, and S. M. Boulaaras, On $\alpha$-fractal functions and their applications to analyzing the S\&P BSE Sensex in India, Chaos, Solitons \& Fractals, 186 (2024) 115194.
\bibitem{K1} S. Kusuoka, A diffusion process on a fractal, Probabilistic methods in mathematical physics, Katata/Kyoto, 1985 (1987).

\bibitem{K2} S. Kusuoka, Dirichlet forms on fractals and products of random matrices, Publ. Res. Inst. Math. Sci., 25(4) (1989) 659-680.

\bibitem{LCP1} R. Lal, S. Chandra, A. Prajapati, Fractal surfaces in Lebesgue spaces with respect to fractal measures and associated fractal operators, Chaos, Solitons \& Fractals, 181 (2024) 114684.
\bibitem{Liu} H. Liu, Sub-self-conformal sets, Scientia Magna, 3(4) (2007) 4-11.
   \bibitem{Hui} H. Liu, Continuity of the Hausdorff dimension for sub-self-conformal sets, Bulletin of Mathematical Analysis \& Applications, 4(1) (2012) 231-238.
\bibitem{Mb} B. Mandelbrot, The fractal geometry of nature, W. H. Freeman and Co., San Francisco, 1977.
\bibitem{Rd} R. D. Mauldin, S. C. Williams, Hausdorff dimension in graph-directed construction, Transactions of the American Mathematical Society, 309(2) (1988) 811-829.
 \bibitem{Mac} M. McClure, R. W. Vallin, The Borel structure of the collections of sub-self-similar sets and super-self-similar sets, Acta Mathematica Universitatis Comenianae, New Series, 69(2) (2000) 145-149.


\bibitem{M2} M. A. Navascu\'es, Fractal polynomial interpolation, Z. Anal. Anwend., 25(2) (2005) 401-418.
\bibitem{Nussbaum1} R. D. Nussbaum, A. Priyadarshi and S. V. Lunel, Positive operators and Hausdorff dimension of invariant sets, Trans. Amer. Math. Soc., 364(2) (2012) 1029–1066.
 
\bibitem{Pr1} A. Priyadarshi, Lower bound on the Hausdorff dimension of a set of complex continued fractions. J. Math. Anal. Appl., 449(1) (2017) 91-95.


\bibitem{Ri2} S.I. Ri, Fractal functions on the Sierpi\'nski gasket, Chaos, Solitons and Fractals, 138 (2020) 110142.
\bibitem{Ruan4} S.G. Ri, H.J. Ruan, Some properties of fractal interpolation functions on Sierpinski gasket, J. Math. Anal. Appl., 380 (2011) 313-322.

\bibitem{Ruan3} H. J. Ruan, Fractal interpolation functions on post critically finite self-similar sets, Fractals, 18 (2010) 119-125.
\bibitem{Sn} N. Snigireva, Inhomogeneous self-similar sets and measures, PhD diss., University of St Andrews, 2008.
\bibitem{SP} A. Sahu, A. Priyadarshi, On the box-counting dimension of graphs of harmonic functions on the Sierpi\'{n}ski gasket, J. Math. Anal. Appl., 487(2) (2020) 124036. 

\bibitem{Shin} J. Shinmoto, F. Takeo. The Hausdorff dimension of sub-self-similar sets, Fractals, 11(1) (2003) 9-18.
\bibitem{Shmerkin} P. Shmerkin, On Furstenberg's intersection conjecture, self-similar measures, and the $L^q$ norms of convolutions, Ann. of Math., 189(2) (2019) 319-391.

\bibitem{MV} M. Verma, Dimensional approximation of inhomogeneous attractor without any separation condition, Bulletin of the Australian Mathematical Society, 112(3) (2025) 515-527.

\bibitem{VM1} S. Verma, P.R. Massopust, Dimension preserving approximation, Aequationes Mathematicae, 96(6) (2022) 1233-1247.
\bibitem{VP1} S. Verma, A. Priyadarshi, Further analysis of Hausdorff dimension and separation conditions, Contemporary Mathematics (AMS), 825 (2025) 225-239.
\bibitem{Byu} B. Yu, Y. Liang, and J. Liu, The intrinsic connection between the fractal dimension of a real number sequence and convergence or divergence of the series formed by it, Journal of Mathematical Analysis and Applications, 539(1) (2024) 128485.
\end{thebibliography}

\end{document}